\documentclass[11pt]{article}
\usepackage{geometry}                
\usepackage{amsmath, amsthm, amssymb}
\usepackage{graphicx, color}
\usepackage{pdfsync}
\newtheorem{lemma}{Lemma}

\newtheorem{proposition}[lemma]{Proposition}
\newtheorem{remark}[lemma]{Remark}

\newtheorem{theorem}[lemma]{Theorem}
\newtheorem{definition}[lemma]{Definition}
\usepackage{graphicx}
\usepackage{amssymb}
\usepackage{epstopdf}
\newcommand\blfootnote[1]{%
  \begingroup
  \renewcommand\thefootnote{}\footnote{#1}%
  \addtocounter{footnote}{-1}%
  \endgroup}
  
\newcommand{\e}{\varepsilon}

\newcommand{\disp}{\displaystyle}

\def\ZZ{\mathbb Z}

\def\eps{\varepsilon}

\title{Asymptotic behaviour of ground states for mixtures\\ of ferromagnetic and antiferromagnetic interactions\\ in a dilute regime}
\author{Andrea Braides\\ \small Dipartimento di Matematica, Universit\`a di Roma Tor Vergata
\\ \small  via della ricerca scientifica 1, 00133 Roma, Italy\\ \\  Andrea Causin\ \\ \small
DADU, Universit\`a di Sassari\\ \small
 piazza Duomo 6, 07041 Alghero (SS), Italy \\ \\ Andrey Piatnitski\\ \small
The Arctic University of Norway, Campus in Narvik,  N-8505 Narvik, Norway
               \\ \small Institute for Information Transmission Problems RAS, \\ \small Bolshoi Karetnyi 19, Moscow, 127051, Russia\\ \\
Margherita Solci \\ \ \small
DADU, Universit\`a di Sassari\\ \small
 piazza Duomo 6, 07041 Alghero (SS), Italy}
\date{}                                           

\begin{document}
\maketitle

\begin{abstract}
We consider randomly distributed mixtures of bonds of ferromagnetic and antiferromagnetic type in a two-dimensional square lattice with probability $1-p$ and $p$, respectively, according to an i.i.d.~random variable.
We study minimizers of the corresponding nearest-neighbour spin energy on large domains in $\ZZ^2$. 
We prove that there exists $p_0$ such that for $p\le p_0$ such minimizers are characterized by a majority phase; i.e., they take identically the value $1$ or $-1$ except for small disconnected sets. A deterministic analogue is also proved.
\end{abstract}


\section{Introduction}
We consider randomly distributed mixtures of bonds of ferromagnetic and antiferromagnetic type in a two-dimensional square lattice with probability $1-p$ and $p$, respectively, according to an i.i.d.~random variable.
For each realization $\omega$ of that random variable, we consider, for each bounded region $D$, the energy
$$
F^\omega(u, D) =-\sum_{i,j}c^\omega_{ij}u_i u_j,
$$
where the sum runs over nearest-neighbours in the square lattice contained in $D$,
$u_i\in \{-1,+1\}$ is a spin variable, and $c^\omega_{ij}\in \{-1,+1\}$ are interaction coefficients
corresponding to the realization.
\begin{figure}[h!]
\centerline{\includegraphics [width=1.7in]{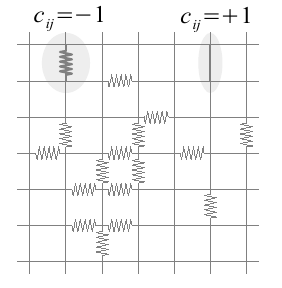}}
\caption{representation of a portion of spin system for some $c_{ij}=c^\omega_{ij}$}\label{FIG1}
   \end{figure}
A portion of such a system is pictured in Fig.~\ref{FIG1}: ferromagnetic bonds; i.e, when $c^\omega_{ij}=1$, are pictured as straight segments, while antiferromagnetic bonds are pictured as wiggly ones (as in the two examples highlighted by the gray regions, respectively).

In this paper we analyze ground states; i.e., absolute minimizers, for such energies. This is a non trivial issue since in general, ground states $\{u_j\}$ are {\em frustrated}\/; i.e., the energy cannot be separately minimized on all pairs of nearest neighbors. In other words, minimizing arrays $\{u_j\}$  may not satisfy simultaneously $u_i=u_j$ for all $i,j$ such that $c^\omega_{ij}=+1$ and $u_i=-u_j$ for all $i,j$ such that $c^\omega_{ij}=-1$. However,
in \cite{BP-hom} it is shown that if the antiferromagnetic links are contained in well-separated compact regions, then the ground states are characterized by a ``majority phase''; i.e., they mostly take only the value $1$ (or $-1$) except for nodes close to the ``antiferromagnetic islands''. In the case of random interactions we show that this is the same in the {\em dilute case}; i.e., when the probability $p$ of antiferromagnetic interactions is sufficiently small. More precisely, we show that there exists $p_0$ such that if $p$ is not greater than
$p_0$ then almost surely for all sufficiently large regular bounded domain $D\subset{\mathbb R}^2$ the minimizers of the energy $F(\cdot,D)$ are characterized by a majority phase.

The proof of our result relies on a scaling argument as follows: we remark that proving the existence of majority phases is equivalent to ruling out the possibility of large interfaces separating zones where a ground state $u$ equals $1$ and $-1$, respectively. Such interfaces may exist only if the percentage of antiferromagnetic bonds on the interface is larger than $1/2$. We then estimate the probability of such an interface with a fixed length and decompose a separating interface into portions of at most that length, to prove a contradiction if $p$ is small enough.

Interestingly, the probabilistic proof outlined above carries on also to a deterministic periodic setting; i.e., for energies
$$
F(u, D) =-\sum_{i,j}c_{ij}u_i u_j
$$
such that $c_{ij}\in \{-1,+1\}$ and there exists $N\in\mathbb N$
such that $c_{i+k\,j+k}= c_{ij}$ for all $i$ and $j\in\mathbb Z^2$ and $k\in N{\mathbb Z}^2$.
In this case ground states of $F$ may sometimes be characterized more explicitly
and exhibit various types of configurations independently
of the percentage of antiferromagnetic bonds: up to boundary effects, there can be
a finite number of periodic textures, or configurations characterized by layers of periodic patterns in one direction, or we might have arbitrary configurations of minimizers with no periodicity (see the examples in \cite{BC}).
We show that there exists $p_0$ such that if the percentage $p$ of antiferromagnetic interactions is not greater than $p_0$ then the proportion of $N$-periodic systems $\{c_{ij}\}$ such that the minimizers of the energy $F(\cdot,D)$ are characterized by a majority phase for all $D\subset{\mathbb R}^2$ bounded domain large enough tends to $1$ as $N$ tends to $+\infty$. The probabilistic arguments are substituted by a combinatorial computation, which also allows a description of the size of the separating interfaces in terms of $N$.

This work is part of a general analysis of variational problems in lattice systems (see \cite{Seoul} for an overview), most results dealing with spin systems focus on ferromagnetic Ising systems
at zero temperature, both on a static framework (see \cite{CDL,AG,BCPS}) and a dynamic framework (see \cite{BGN,BSc,BSc2,BSo}). In that context, random distributions of bonds have been considered in \cite{BP-hom,BP-dil} (see also \cite{BP-mem}), and their analysis is linked to some recent advances in Percolation Theory
(see \cite{Boivin,CePi1,CeTe,Garet,W}).
 A first paper dealing with antiferromagnetic interactions is \cite{ABC}, where non-trivial oscillating ground states are observed and the corresponding surface tensions are computed. A related variational motion of crystalline mean-curvature type has been recently described in \cite{BCY}, highlighting new effect due to surface microstructure. The classification of periodic systems mixing ferromagnetic and antiferromagnetic interactions that can be described by surface energies is the subject of \cite{BC}.
In \cite{BP-hom}, as mentioned above, the case of well-separated antiferromagnetic island is studied.
We note that in those papers the analysis is performed by a description of a macroscopic surface tension,
which provides the energy density of a continuous surface energy obtained as a discrete-to-continuum $\Gamma$-limit \cite{GCB} obtained by scaling the energy $F$ on lattices with vanishing lattice space.
In the present paper we do not address the formulation in terms of the $\Gamma$-limit but only
study ground states.

\section{Random media}

Given a probability space $(\Omega,\mathcal{F},\mathbf{P})$ we
consider a {\em Bernoulli bond percolation} model in $\mathbb Z^2$. This means that to each bond $(i,j)$, $i,j\in\mathbb Z^2$, $|i-j|=1$, in $\mathbb Z^2$  we associate a random variable $c_{ij}$ and assume that
these random variables are i.i.d.~and that they take on the value  $+1$ with probability $1-p$, and the value $-1$ with probability $p$, where $0<p<1$. The detailed description of the Bernoulli bond percolation model can be found for instance in \cite{Grimm}

We denote by $\mathcal N$ the set of {\em nearest neighbors}
$$\mathcal N=\{\{i,j\}: i,j\in \mathbb Z^2, \|i-j\|=1\}$$
and, for each $\{i,j\}$ in $\mathcal N$, $[i,j]$ will be the closed segment with endpoints $i$ and $j$.

\begin{definition}[random stationary spin system]\label{def-random} A (ferromagnetic/antiferromagnetic)
{\em spin system} is a realization of the random function $c(\{i,j\})=c_{ij}(\omega)\in\{\pm1\}$ defined on
$\mathcal N$. We will drop the dependence on $\omega$ and simply write $c_{ij}$.
The pairs $\{i,j\}$ with $c_{ij}=+1$ are called {\em ferromagnetic
bonds}, the pairs $\{i,j\}$ with $c_{ij}=-1$ are called {\em antiferromagnetic
bonds}.
\end{definition}

\subsection{Estimates on separating paths}

We say that a finite sequence
$(i_0,\ldots, i_k)$ 
is a {\em path} in $\mathbb Z^2$ 
if
 $\{i_s,i_{s+1}\}\in\mathcal N$
 for any $s=0,\ldots, k-1$ and the segment $[i_s,i_{s+1}]$
 is different from the segment $[i_t,i_{t+1}]$
 for any $s\neq t$.
 The path is {\em closed} if $i_0=i_k$. The number $k$ is the {\em length} of $\gamma$, denoted by $l(\gamma)$, and we call
 $\mathcal P_k$ the set of the {\em paths with length $k$}.
 To each path $\gamma\in\mathcal P_k$ we associate
the corresponding curve $\tilde\gamma$ of length $k$
in $\mathbb R^2$ given by
\begin{equation}\label{tildegamma}
\tilde\gamma=\bigcup_{s=0}^{k-1}[i_s, i_{s+1}]+\Big(\frac{1}{2}, \frac{1}{2}\Big).
\end{equation}

\begin{figure}[h!]
\centerline{\includegraphics [width=4in]{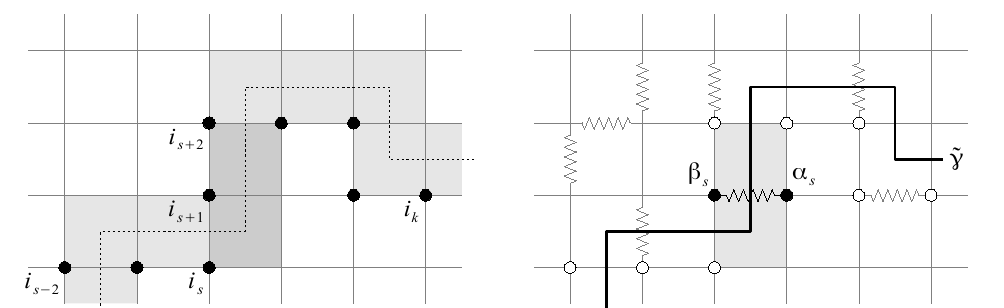}}
\caption{a path $\gamma$ and the corresponding curve $\tilde \gamma$}\label{FIG2}
   \end{figure}
Note that $\tilde\gamma$ is a closed curve if and only if $\gamma$ is closed.
In Fig.~\ref{FIG2} we picture a path (the dotted sites of the left-hand side) and
the corresponding curve (on the right-hand side picture).

Given two paths $\gamma=(i_0,\dots , i_k)$ and $\delta=(j_0,\dots , j_h)$, if $i_k=j_0$ and the sequence $(i_0,\dots , i_k,j_1,\dots , j_h)$ is a path, the latter is called the {\em concatenation} 
 of $\gamma$ and $\delta$ and it is noted by $\gamma\ast\delta$.
 %

We note that for each $s$ the intersection $(i_s+[0,1]^2)\cap (i_{s+1}+[0,1]^2)$
is a segment with endpoints $\{\alpha_s, \beta_s\}\in \mathcal N$;
then, given a spin system $\{c_{ij}\}$, for each path $\gamma=(i_0,\ldots,i_k)$ we can define
the {\em number of antiferromagnetic bonds} of $\gamma$ as
\begin{equation}\label{mu}
\mu(\gamma)=\mu(\gamma, \{c_{ij}\})=\#\{s\in \{0,\dots, k-1\}: c_{\alpha_s\beta_s}=-1\}.
\end{equation}
If $\tilde\gamma$ is the curve corresponding to $\gamma$ defined above,
then the number $\mu(\gamma)$ counts the antiferromagnetic interactions ``intersecting'' $\tilde\gamma$ (see Fig.~\ref{FIG2}).


\begin{definition}[Separating paths]\label{sep_bis}
A path $\gamma$ of length $k$ is a {\em separating path}
for a spin system $\{c_{ij}\}$ if $\mu(\gamma)>k/2$.
\end{definition}

\begin{remark}\rm The terminology separating path evokes the fact that only
closed separating
paths may enclose (separate) regions where a minimal $\{u_i\}$ is constant.
Indeed, if we have $u_i=1$ on a finite set $A$
of nodes in $\mathbb Z^2$ which is connected (i.e., for every pair $i,j$ of points in
$A$ there is a path of points in $A$ with $i$ as initial point and $j$ as final point)
and $u_i=-1$ on all neighbouring nodes, then the boundary of $A$ (i.e., the set of
points $i\in A$ with a nearest neighbour not in $A$) determines a path. If such a
path is not separating then the function $\tilde u$ defined as $\tilde u_i=-u_i$ for $i\in A$
and $\tilde u_i=u_i$ elsewhere has an energy strictly lower than $u$.
\end{remark}

\begin{remark}\rm
For a path $\gamma$ of length $l(\gamma)=k$ the probability that $\gamma$ be separating
can be estimated as follows
\begin{equation}\label{prob_est1}
\mathbf{P}\{\mu(\gamma)>k/2\}\leq p^{k/2}2^k
\end{equation}
Indeed, the probability that $c_{ij}$ is equal to ${-1}$ at $k/2$ fixed places is equal to
$p^{k/2}$. Since $\Big(\begin{array}{c}k\\k/2\end{array}\Big)$ does not exceed $2^k$, the desired
estimates follows.
\end{remark}

\begin{lemma}\label{l_proba1}
There exists $p_0>0$ such that for any $\varkappa>0$ and for all $p<p_0$ almost surely
for sufficiently large $n$ in a cube $Q_n=[0,n]^2$ there is no a separating path $\gamma$
with $l(\gamma)\geq (\log(n))^{1+\varkappa}$.
\end{lemma}

\begin{proof} We use the method that in percolation theory often called "path counting" \ argument.
The number of paths of length $k$ starting at the origin is not greater than $3^k$.
Therefore, in view of \eqref{prob_est1}  the probability that there exists a separating path
of length $k$ that starts at the origin is not greater than $p^{k/2}2^k3^k$.
Letting $p_0=(1/12)^2$ we have
$$
p^{k/2}2^k 3^k\leq 2^{-k}\qquad \hbox{for all }p\leq p_0.
$$
Then, if $p\leq p_0$, the probability that there exists a separating path of length
$k$ in a cube $Q_n$ does not exceed $n^2 2^{-k}$.
For $k\geq\log(n)^{1+\varkappa}$ this yields
$$
\mathbf{P}\{\hbox{there exists a separating path }\gamma\subset Q_n\hbox{ of length }k\}
$$
$$
\leq n^2 2^{-\log(n)^{1+\varkappa}}=n^{2-c_1\log(n)^\varkappa}
$$
with $c_1=\log 2$.
Finally, summing up in $k$ over the interval $[\log(n)^{1+\varkappa},n^2]$
we obtain
$$
\mathbf{P}\{\hbox{there exists a separating path }\gamma\subset Q_n\hbox{ such that }
l(\gamma)\geq \log(n)^{1+\varkappa}\}
$$
$$
\leq n^{4-c_1\log(n)^\varkappa}
$$
Since for large $n$ the right-hand side here decays faster than any negative power of $n$,
the desired statement follows from the Borel-Cantelli lemma.
\end{proof}

\subsection{Geometry of minimizers in the random case}

Let $D$ be a bounded open subset of $\mathbb R^2$ 
and $u\colon D\cap\mathbb Z^2\to\{\pm 1\}$.
Then, denoting by $\mathcal N(D)$ the set of nearest neighbors in $D$, $F(u,D)$ is defined by
\begin{equation}\label{effe_bis}
F(u,D)=-\sum_{\{i,j\}\in \mathcal N(D)}c_{ij}u_iu_j.
\end{equation}
Note that the energy depends on $\omega$ through $c_{ij}$. We will characterize the almost-sure behaviour of ground states for such energies.

We define the {\em interface} $S(u)$ as
$$S(u)=S(u;D)=\{\{i,j\}\in\mathcal N(D):
u_iu_j=-1\};$$
we associate to each pair $\{i,j\}\in S(u)$ the segment $s_{ij}=\overline Q_i\cap \overline Q_j$,
where $Q_i$ is the coordinate unit open square centered at $i$, and consider the
set
\begin{equation}\label{esse}
\Sigma(u)=\Sigma(u;D)=\bigcup_{\{i,j\}\in S(u)} s_{ij}.
\end{equation}
If we extend the function $u$ in $\bigcup_{i\in D\cap\mathbb Z^2} Q_i$
by setting $u=u_i$ in $Q_i$, and define
\begin{equation}\label{q}
q(D)=\hbox{int}\Big(\bigcup_{i\in D\cap\mathbb Z^2} \overline Q_i\Big),
\end{equation}
then the set $\Sigma(u)\cap 
q(D)$
turns out to be the jump set of $u$ 
and we can write
$$\Sigma(u)=\partial \overline{\{u=1\}}\cap \partial \overline{\{u=-1\}}.$$

In the following remark we recall some definitions and classical results related to the notion of graph which will be useful to establish  properties of the connected components of $\partial\overline{\{u=1\}}$. For references on this topic, see for instance \cite{graph}.

\begin{remark}[Graphs and two-coloring]\label{verde_bis} \rm
We say that a triple $G=(V,E,r)$ is a {\em multigraph} when $V$ ({\em vertices}) and $E$ ({\em edges}) are finite sets and
$r$ ({\em endpoints}) is a map from $E$ to $V\otimes V$, where $\otimes$ denotes the symmetric product.
The {\em order} of a vertex $v$ is
$\#\{e\in E: r(e)= x\otimes v \hbox{ for some } x\in V\}+\#\{e\in E: r(e)=v\otimes v\},$ so that the loops are counted twice.
A {\em walk} in the graph $G$ is a sequence of
edges $(e_1,\dots, e_n)$ such that
there exists a sequence of vertices $(v_0,\dots, v_n)$ with the property $r(e_i)=v_{i-1}\otimes v_{i}$ for each $i$;
if moreover 
$v_n=v_0$,
then the walk is called  a {\em circuit}.
The multigraph $G$ is {\em connected} if given $v\neq v^\prime$ in $V$ there exists a walk connecting them, that is a walk such that $v_0=v$ and $v_n=v^\prime$ in the corresponding sequence of vertices.

We say that $G$ is {\em Eulerian} if there is a circuit containing every element of $E$ exactly once (Eulerian circuit). A classical theorem of Euler (see \cite[Ch.~3]{graph} and \cite{euler} for the original formulation) states that $G$ is Eulerian if and only if $G$ is connected and the order of every vertex is even.

A multigraph $G$ is {\em embedded in $\mathbb R^2$} if $V\subset\mathbb R^2$
and the edges are simple curves in $\mathbb R^2$ such that the endpoints belong to $V$ and
two edges can only intersects at the endpoints.
An embedded graph is Eulerian if and only if the union of the edges $\bigcup_{e\in E} e$ is connected, and
its complementary in $\mathbb R^2$ can be {\em two-colored}, that is
$\mathbb R^2\setminus \bigcup_{e\in E} e$
is the union of two disjoint sets $B$ and $W$ such that $\partial B=\partial W=\bigcup_{e\in E} e$.
\end{remark}

\begin{remark}[Eulerian circuits in $\partial \overline{\{u=1\}}$]\label{euler_bis}\rm
Let $C$ be a connected component of $\partial \overline{\{u=1\}}$. 
We can see $C$ as a connected embedded graph whose vertices are the points in $(\mathbb Z^2+(1/2,1/2))\cap C$ and
two vertices share an edge if there is a unit segment in $C$ connecting them.
By construction, $\mathbb R^2\setminus C$ can be {\em two-colored}, hence
$C$ is an Eulerian circuit (see Remark \ref{verde_bis}). Recalling the definition of path and the definition $(\ref{tildegamma})$, this corresponds to say that there exists a closed path $\eta$ 
such that $\tilde\eta=C$.
\end{remark}

We say that a path $\gamma\in\mathcal P_k$ {\em is in the interface} $S(u)$ if the corresponding $\tilde\gamma\subset \Sigma(u)$.

\smallskip
Let $G$ be a Lipschitz bounded domain in $\mathbb R^2$.
\begin{theorem}\label{th_rand}   Let $p<p_0$, and let $u_\eps$ be a minimizer for
$F(\cdot,\frac1\eps G)$. Then for any $\varkappa>0$ almost surely for all sufficiently small $\eps>0$
 either $\overline{\{u_\eps = 1\}}$ or $\overline{\{u_\eps = -1\}}$ is composed of connected
 components $K_i$
such that the length of the boundary of each $K_i$
is not greater than $|\log(\eps)|^{1+\varkappa}$.
\end{theorem}

\begin{proof}
We say that a path $\gamma\in\mathcal P_k$ {\em is in the interface} $S(u)$ if the corresponding $\tilde\gamma\subset \Sigma(u)$.
The proof of Theorem essentially relies on the following statement.
\begin{proposition} \label{path_bis}
For any $\Lambda>0$ a.s.~for sufficiently small $\eps>0$ and for any  open bounded subset $D\subset(-\Lambda/\eps,\Lambda/\eps)^2$ such that 
the distance between the connected components of $\partial q(D)$ is greater than $|\log(\eps)|^{1+\varkappa}$
for a minimizer $u$ of $F(\cdot,D)$
there is no path in the interface $S(u)$ of length greater than $|\log(\eps)|^{1+\varkappa}$.
\end{proposition}

\begin{proof}
Let $\gamma\in\mathcal P_k$ be a path in the interface $S(u)$ with $k\geq |\log(\eps)|^{1+\varkappa}$.
We denote by $C$ the connected component of $\partial \overline{\{u=1\}}$
containing $\tilde \gamma$. Remark \ref{euler_bis} ensures that
$C=\tilde\eta$ where $\eta$ is a closed path; hence, up to extending $\gamma$ in $\eta$,
we can assume without loss of generality that
$\gamma$ is a path of maximal length in the interface $S(u)$.

We start by showing that
there exists a closed path $\sigma=\sigma(\gamma)$ such that $\tilde\gamma\subset \tilde\sigma\subset
C$ satisfying the following property:
\begin{equation} \label{sigma_bis}
\left. \begin{array}{ll}
 \ r \hbox{ different paths } \eta_1,\dots, \eta_r \hbox{ exist such that }
&\tilde \sigma \cap \Sigma(u)
\vspace{1.5mm}=\bigcup_{t}\tilde\eta_t \\
&l(\eta_t)\geq |\log(\eps)|^{1+\varkappa}/2  \hbox{ for all }t.
\end{array}
\right.
\end{equation}
If $\gamma$ is closed, then we set $\sigma=\gamma=\eta_1$.
Otherwise, $\tilde\gamma$ connects two points in $\partial q(D)$.

If these endpoints belong to the same connected
component 
of $\partial q(D)$, then we can choose a path $\delta$ such that $\tilde \delta$ lies in $\partial q(D)$ and has the same endpoints of $\tilde\gamma$ and, recalling the notion of concatenation of paths, we can define $\sigma$ as $\gamma\ast\delta$, and again $\gamma=\eta_1$.

It remains to construct $\sigma$ when the endpoints of $\tilde\gamma$
belong to different connected components of $\partial q(D)$.
We consider the set $V$ of the connected components of $\partial q(D)$ and the set $E$
of the connected components of $\tilde\eta\cap\Sigma(u)$ (note that $\tilde\gamma\in E$).
By the existence of the path $\eta$,
each element of $E$ is a curve connecting
two (possibly equal) elements of $V$, then $(V,E)$ is a multigraph.
Since $\tilde\eta$ is a closed curve containing $\tilde\gamma$, it realizes in the graph an Eulerian
circuit containing $\tilde\gamma$. Therefore, there exists a minimal Eulerian circuit
$(\tilde\gamma=\tilde\eta_1,\tilde\eta_2,\dots, \tilde\eta_r)$ and,
by minimality, the order of each vertex touched by this circuit is $2$ (see Remark \ref{verde_bis}).
Denoting by $\Delta_t$ the vertex shared by $\tilde\eta_t$ and $\tilde\eta_{t+1}$ for $t<r$, and
by $\Delta_r$ the vertex shared by $\tilde\eta_r$ and $\tilde\eta_1$, for each $t$
we can find a path $\delta_t$ such that $\tilde\delta_t\subset\Delta_{t}$ and such that
the path $\sigma=\delta_1\ast\eta_1\ast\dots\ast\delta_r\ast\eta_r$ is closed and satisfies
the property (\ref{sigma_bis}).

\smallskip

Since $\sigma$ is a closed path, then $\tilde\sigma$ is a closed properly self-intersecting curve so that
all the vertices of the corresponding embedded graph have even order.
Remark \ref{verde_bis} ensures that the embedded graph corresponding to $\tilde\sigma$ is Eulerian, hence
its complementary $\mathbb R^2\setminus \tilde\sigma$ can be two-colored, that is it is the union of two disjoint sets $B$ and $W$ such that $\partial B=\partial W=\tilde\sigma$. Setting $\tilde u$ as the extension to $\cup_{i\in D\cap\mathbb Z^2} Q_i$ of the function
$$
\tilde u_i=\left\{ \begin{array}{ll}
u_i & \hbox{ in } B\cap D\cap\mathbb Z^2\\
-u_i & \hbox{ in } W\cap D\cap\mathbb Z^2
\end{array}
\right.
$$
it follows that
$F(u,D)-F(\tilde u,D)=2\sum_{t=1}^n(l(\eta_t)-2\mu(\eta_t))$, where $\mu(\eta_t)$ stands for the number of antiferromagnetic interactions in $\eta_t$ as defined in (\ref{mu}).
Since $u$ minimizes $F(\cdot, D)$, we can conclude that, for at least one index $t$, $\mu(\sigma_t)\geq l(\sigma_t)/2$; that is, $\sigma_t$ is a separating path of length greater than $|\log(\eps)|^{1+\varkappa}$, contradicting Lemma \ref{l_proba1} and concluding the proof.
\end{proof}

We turn to the proof of Theorem \ref{th_rand}. Letting $G_\e=q(\frac{1}{\e}G),$ we
consider the connected components 
of the interface $\Sigma(u_\e)$.
Since each of them 
corresponds to a path, they are either closed curves, denoted by $C_\e^i$ for $i=1,\dots, n$,
or curves with the endpoints in $\partial G_\e$, denoted by $D_\e^j$ for $j=1,\dots, m$.
Since for $\e$ small enough the distance between two connected components of $\partial G_\e$ is greater than $|\log(\e)|^{1+\varkappa}$, Proposition \ref{path_bis} ensures that in both cases the length of such curves is less than $|\log(\e)|^{1+\varkappa}$.

The distance between the endpoints of a component $D_\e^j$
is less than $|\log(\e)|^{1+\varkappa}$, and, since $G$ is Lipschitz,
for $\e$ small enough we can find a path in $\partial G_\e$ with the same endpoints and length less than $\tilde C |\log(\e)|^{1+\varkappa}$.
This gives a closed path $S_\e^j$ with length less than $\tilde C |\log(\e)|^{1+\varkappa}$ containing $D_\e^j$.

The set $\mathbb R^2\setminus \left(\bigcup_{i}C_\e^i\cup\bigcup_jD_\e^j\right)$ has exactly one
unbounded connected component, which we call $P_\e$.
The function $u_\e$ is constant in $P_\e\cap G_\e$. Assuming that this constant value is $1$,
then $\partial\overline{\{u_\e=-1\}}$ is contained in $\left(\bigcup_{i}C_\e^i\cup\bigcup_jD_\e^j\right)$ and
the boundary of every connected component $K_\e^i$
of $\overline{\{u_\e=-1\}}$ has length less than $C|\log(\e)|^{1+\varkappa}$.
\end{proof}

\section{Periodic media}
We now turn our attention to a deterministic analog of the problem discussed above, where random
coefficients are substituted by periodic coefficients and the probability of having antiferromagnetic interactions is
replaced by their percentage.

\subsection{Estimates on the number of antiferromagnetic interactions along a path}
In order to prove a deterministic analogue of Theorem \ref{th_rand}, we need to give an estimate of the length of separating paths corresponding to the result stated in Lemma \ref{l_proba1}.
We start with the definition of a periodic spin system in the deterministic case given on the lines of Definition \ref{def-random}.
\begin{definition}[periodic spin system]With fixed $N\in\mathbb N$, a deterministic (ferromagnetic/antiferromagnetic)
{\em spin system} is a function $c(\{i,j\})=c_{ij}\in\{\pm1\}$ defined on
$\mathcal N$. The pairs $\{i,j\}$ with $c_{ij}=+1$ are called {\em ferromagnetic
bonds}, the pairs $\{i,j\}$ with $c_{ij}=-1$ are called {\em antiferromagnetic
bonds}.
We say that a spin system is {\em $N$-periodic} if
$$c(\{i,j\})=c(\{i+(N,0),j+(N,0)\})=c(\{i+(0,N),j+(0,N)\}).$$
\end{definition}

In the sequel of this section, when there is no ambiguity we use the same terminology and notation concerning the random case given in Section 2.

\begin{definition}[spin systems with given antiferro proportion]
For $p\in (0,1)$ we consider the set $\mathcal C_p(N)$ of $N$-periodic spin systems $\{c_{ij}\}$ such that
the number of antiferromagnetic interactions in $[0,N]^2$ is $\lfloor 2pN^2\rfloor$, and for any
$\lambda\in(0,1)$ we define
\begin{equation}\label{def-conf-cattive}
\mathcal B^\lambda_p(N)=\left\{\{c_{ij}\}\in \mathcal C_p(N): \exists \ \gamma\in \mathcal P_k(N)
\hbox{ separating path for $\{c_{ij}\}$ with $k\geq \lambda N$}\right\}
\end{equation}
where
$\mathcal P_k(N)$ is the set of paths $\gamma=(i_0,\dots,i_k)\in \mathcal P_k$ such that
$i_s\in [0,N]^2$ for each $s.$
\end{definition}

\begin{proposition}
There exists $p_0>0$ such that for every $p<p_0$ and $\lambda>0$:
$$\lim_{N\to +\infty}
\frac{\# \mathcal B^\lambda_p(N)}{\# \mathcal C_p(N)}=0.$$
\end{proposition}
\begin{proof}
We fix a path $\gamma\in\mathcal P_k(N)$ with $\lambda N\leq k\leq 4p_NN^2$, where
$p_N=\frac{\lfloor 2pN^2\rfloor}{2N^2}$.
Then, the number of spin systems $\{c_{ij}\}$ in $\mathcal C_p(N)$
for which $\gamma$ is a separating path depends only on $k$ and it is given by
$$f_{p}(k,N)=\sum_{j=k/2}^{\min\{k,2p_N N^2\}}
\begin{pmatrix}k\\j\end{pmatrix}
\begin{pmatrix}2N^2-k\\2p_N N^2-j\end{pmatrix}.$$
Since
\begin{eqnarray*}
\# \mathcal B^\lambda_p(N)&\leq&\sum_{k=\lfloor\lambda N\rfloor}^{4p_N N^2}\# \{
\mathcal P_k(N)\}f_{p}(k,N) \ \leq \ \sum_{k=\lfloor\lambda N\rfloor}^{4p_N N^2}3^k N^2 f_{p}(k,N)\\
\# \mathcal C_p(N)&=&\begin{pmatrix}2N^2\\
2p_N N^2\end{pmatrix}
\end{eqnarray*}
we get the estimate
$$\frac{\# \mathcal B^\lambda_p(N)}{\# \mathcal C_p(N)}\leq \sum_{m=m(\lambda)}^{m(N)}\left(\sum_{k=2^mN}^{2^{m+1}N}3^k N^2 f_{p}(k,N) \begin{pmatrix}2N^2\\
2p_N N^2\end{pmatrix}^{-1}\right)$$
where $m(\lambda)=\lfloor\log_2(\lambda)\rfloor-1$ and $m(N)=\lfloor \log_2 (4p_N N)\rfloor -1$.
Noting that
$$\begin{pmatrix}2p_N N^2\\j\end{pmatrix}
\begin{pmatrix}2(1-p_N) N^2\\k-j\end{pmatrix}
\leq \begin{pmatrix}2p_N N^2\\k/2\end{pmatrix}
\begin{pmatrix}2(1-p_N) N^2\\k/2\end{pmatrix}
$$
for each $j=k/2, \ldots, \min\{k,2p_N N^2\}$,
we get
\begin{equation}\label{effe}
\left. \begin{array}{ll}
\disp f_{p}(k,N)\begin{pmatrix}2N^2\\
2p_N N^2\end{pmatrix}^{-1}&=\disp\sum_{j=k/2}^{\min\{k,2p_N N^2\}}
\begin{pmatrix}k\\j\end{pmatrix}
\begin{pmatrix}2N^2-k\\2p_N N^2-j\end{pmatrix}\begin{pmatrix}2N^2\\
2p_N N^2\end{pmatrix}^{-1}\\
&\disp=\sum_{j=k/2}^{\min\{k,2p_N N^2\}}
\begin{pmatrix}2p_N N^2\\j\end{pmatrix}
\begin{pmatrix}2(1-p_N) N^2\\k-j\end{pmatrix}
\begin{pmatrix}2N^2\\k\end{pmatrix}^{-1}\\
&\disp\leq\Bigl(\min\{k,2p_N N^2\}-\frac{k}{2}\Bigr)
g_{p}(k,N)
\end{array}
\right.
\end{equation}
where
$$g_{p}(k,N)=\begin{pmatrix}2p_N N^2\\k/2\end{pmatrix}
\begin{pmatrix}2(1-p_N) N^2\\k/2\end{pmatrix}
\begin{pmatrix}2N^2\\k\end{pmatrix}^{-1}.
$$
Now, we prove an estimate for $g_{p}(k,N)$.
\begin{lemma}\label{g}
For any $k,N\in\mathbb N$ such that $k\leq 4pN^2$ and for any $p\in(0,1/2)$ we have
$$g_{p}(k,N)\leq C(p)N^8(\theta(p))^{k}$$
with $\theta(p)=2 e \sqrt{p(1-p)}$.
\end{lemma}
\begin{proof}[Proof of Lemma {\rm\ref{g}}]
Recalling that $n^ne^{1-n}\leq n! \leq n^{n+1} e^{1-n}$ for any $n$ in $\mathbb N$, we get
\begin{equation*}
\begin{pmatrix}2N^2\\k\end{pmatrix}^{-1}=
\frac{k!(2N^2-k)!}{(2N^2)!}
\leq e k (2N^2-k)\left(\frac{k}{2N^2-k}\right)^k
\left(\frac{2N^2-k}{2N^2}\right)^{2N^2}
\end{equation*}
and for $k\neq 4p_NN^2$
\begin{eqnarray*}
\begin{pmatrix}2p_NN^2\\k/2\end{pmatrix}&=&
\frac{(2p_NN^2)!}{(k/2)! (2p_NN^2-k/2)!}\\
&\leq& \frac{2p_NN^2}{e}
\left(2p_NN^2-\frac{k}{2}\right)^{k/2}\left(\frac{k}{2}\right)^{-k/2}
\left(\frac{2p_NN^2}{2p_NN^2-k/2}\right)^{2p_NN^2}.
\end{eqnarray*}
Hence, the following estimate holds
\begin{eqnarray*}
g_{p}(k,N)&\leq&
\frac{4p_N(1-p_N)}{e}N^4 (2N^2-k) k
\left(2 \frac{(2p_NN^2-\frac{k}{2})^{1/2} (2(1-p_N)N^2-\frac{k}{2})^{1/2}}{2(1-p_N)N^2}\right)^{k} \\
&&\left(\frac{2N^2-k}{2N^2}\right)^{2N^2}
\left(\frac{2p_NN^2}{2p_NN^2-k/2}\right)^{2p_NN^2}
\left(\frac{2(1-p_N)N^2}{2(1-p_N)N^2-k/2}\right)^{2(1-p_N)N^2}.
\end{eqnarray*}
Recalling the inequalities
$$\left(\frac{x-a}{x}\right)^{a} \ e^{-a}\leq \left(\frac{x-a}{x}\right)^x \leq  \left(\frac{x-a}{x}\right)^{a}$$
for $a,x$ such that $0<a<x$, we get
for any $k\neq 4p_NN^2$
\begin{eqnarray*}
&\disp\left(\frac{2N^2-k}{2N^2}\right)^{2N^2}&\leq \hspace{2mm}
\left(\frac{2N^2-k}{2N^2}\right)^{k} \\
&\disp\left(\frac{2p_N N^2}{2p_N N^2-k/2}\right)^{2p_N N^2}&\leq \hspace{2mm}
\left(\frac{2p_N N^2}{2p_N N^2-k/2}\right)^{k/2} e^{k/2}\\
&\disp\left(\frac{2(1-p_N) N^2}{2(1-p_N) N^2-k/2}\right)^{2(1-p_N) N^2}&\leq \hspace{2mm}
\left(\frac{2(1-p_N) N^2}{2(1-p_N) N^2-k/2}\right)^{k/2} e^{k/2}.
\end{eqnarray*}
Since $p_N(1-p_N)\leq p(1-p)$ for $p\in(0,1/2)$, the previous estimates give
\begin{eqnarray*}
g_{p}(k,N)&\leq& 
\frac{4p(1-p)}{e}N^4 (2N^2-k) k \left(2e\sqrt{p_N(1-p_N)}\right)^{k}\\
&\leq& \frac{32p^2(1-p)}{e}N^8 \left(2e\sqrt{p_N(1-p_N)}\right)^{k}\\
&\leq& \frac{32p^2(1-p)}{e}N^8 \left(2e\sqrt{p(1-p)}\right)^{k}
\end{eqnarray*}
concluding the proof for $k\neq 4p_NN^2$. Note that
$\theta(p)= 2 e  \sqrt{p(1-p)} \to 0$ for $p\to 0$.

It remains to check the case $k=4p_NN^2$.  Noting that for $p<1/2$
\begin{eqnarray*}
g_{p}(4p_N N^2,N)&=&\frac{(4p_NN^2)! (2(1-p_N)N^2)! }{(2p_NN^2)!(2N^2)!}\\
&\leq& 8p_N(1-p_N) N^4 (1-p_N)^{2N^2}\left(\frac{2p_N}{\sqrt{p_N(1-p_N)}}\right)^{4p_NN^2}\\
&\leq& 8p(1-p) N^4 \left(2e\sqrt{p(1-p)}\right)^{4p_NN^2}.
\end{eqnarray*}
the thesis of Lemma \ref{g} follows.
\end{proof}
Now, Lemma \ref{g} allows to conclude the proof of the proposition. Indeed,
applying the estimate on $g_p(k,N)$, we get from inequality (\ref{effe})
%
\begin{eqnarray*}\sum_{k=2^mN}^{2^{m+1}N}3^k N^2 f_{p}(k,N) \begin{pmatrix}2N^2\\
2p_NN^2\end{pmatrix}^{-1}
&\leq& pC(p)N^{12}\sum_{k=2^mN}^{2^{m+1}N}(3\theta(p))^{k}\\
&=&pC(p)N^{12}\sum_{t=2^m}^{2^{m+1}}((3\theta(p))^N)^{t}\\
&=&pC(p)N^{12}((3\theta(p))^N)^{2^m} \frac{1-  (3\theta(p))^{(2^m+1)N}}{1-(3\theta(p))^{N}}\\
&\leq& C(p)N^{12}
(3\theta(p))^{2^mN}
\end{eqnarray*}
for $p<1/2$ and for $N$ large enough (independent on $m$).

By summing over $m$, we get
\begin{equation}\label{stima}
\left. \begin{array}{ll}
\disp\frac{\# \mathcal B^\lambda_p(N)}{\# \mathcal C_p(N)}&\disp\leq C(p)N^{12}
\sum_{m=m(\lambda)}^{m(N)} (3\theta(p))^{2^mN}
\\
&\disp\leq C(p)N^{12}
(3\theta(p)^N)^{2^{m(\lambda)}-1}
\sum_{m=m(\lambda)}^{m(N)} (3\theta(p)^N)^{2^m-2^{m(\lambda)}+1} \\
&\disp\leq  C(p)N^{12}(3\theta(p))^{(2^{m(\lambda)}-1)N}\sum_{t=1}^{+\infty} ((3\theta(p))^N)^t\\
&\disp\leq  C(p)N^{12}\frac{(3\theta(p))^{2^{m(\lambda)}N}}{1-(3\theta(p))^{N}}\\
&\disp\leq  2C(p)N^{12}((3\theta(p))^{2^{m(\lambda)}})^N
\end{array}
\right.
\end{equation}
which goes to $0$ as $N\to +\infty$ if
$3\theta(p)=6 e \sqrt{p(1-p)}<1$.
\end{proof}

\begin{remark}[Translations]\label{trans}\rm
Denoting by $\tilde{\mathcal B}_p^\lambda(N)$ the set of $N$-periodic spin systems $\{c_{ij}\}$ such that
there exists a separating path for $\{c_{ij}\}$ in $z+[0,N]^2$ for some $z\in \mathbb Z^2$, then the estimate (\ref{stima})
implies
$$\lim_{N\to +\infty}\frac{\#\tilde{\mathcal B}_p^\lambda(N)}{\#{\mathcal C_p}(N)}=0.$$
\end{remark}

Now, we state the deterministic analogue of Lemma \ref{l_proba1}.
\begin{proposition}\label{globale}
If the $N$-periodic spin system $\{c_{ij}\}$ belongs to $\mathcal C_p\setminus \tilde{\mathcal B}_p^{1/2}(N)$,
then there is no separating path in $\mathbb Z^2$ of length greater than $N/2$.
\end{proposition}

\begin{figure}[h!]
\centerline{\includegraphics [width=5.5in]{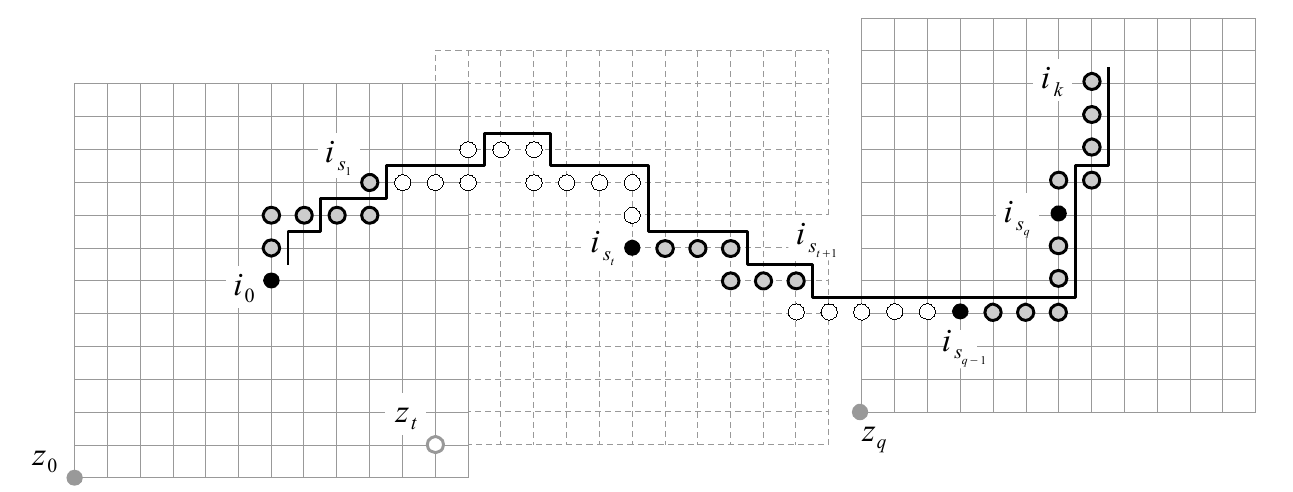}}
\caption{decomposition of $\gamma$}\label{FIG3}
   \end{figure}
\begin{proof}
Let $\gamma=(i_0,\ldots, i_k)$ be a path in $\mathcal P_k$ with $k\geq N/2$.
We decompose $\gamma$ as a concatenation of paths $\gamma_1\ast\dots\ast\gamma_{q-1}\ast\gamma_q$ with
$l(\gamma_t)\geq N/2$ and each $\gamma_t$ contained in a coordinate square
$z+[0,N]^2$ for some $z\in\mathbb Z^2$ (see Fig.~\ref{FIG3}).

If $N$ is even, setting  $q=\lfloor \frac{2k}{N}\rfloor$ and $s_t=tN/2$ for $t=0,\dots,q$, we define
\begin{equation}\label{gamma-t}
\gamma_t=(i_{s_{t-1}},\dots, i_{s_t})
\hbox{ \ \ for }  t=1,\dots,q-1 \hbox{ \ \ and \ \ }
\gamma_{q}=(i_{s_{q-1}},\ldots, i_k). 
\end{equation}
In this way, setting
\begin{eqnarray*}
&&z_t=i_{s_{t-1}}-\left(\frac{N}{2},\frac{N}{2}\right)
\hbox{ \ \ for }  t=1,\dots,q-1\\
&&z_{q}=i_{s_{q}}-\left(\frac{N}{2},\frac{N}{2}\right),
\end{eqnarray*}
it follows that for any $t=1,\dots q$ $\gamma_t$ is a path of length $l(\gamma_t)$ greater than $ N/2$ contained in
$z_t+[0,N]^2$.
Since $\{c_{ij}\}\not\in\tilde{\mathcal B}_p^{1/2}(N)$, the number of antiferromagnetic interactions
$\mu(\gamma_t)$ is less than $l(\gamma_t)/2$ for any $t$. Hence $\mu(\gamma)\leq k/2$.

If $N$ is odd, we pose $q=\lfloor \frac{2k}{N+1}\rfloor$ and $s_t=t(N+1)/2$ for $t=0,\dots,q$;
defining the adjacent paths $\gamma_t$ as in (\ref{gamma-t}),
 by setting
\begin{eqnarray*}
&&z_t=i_{s_{t-1}}-\left(\frac{N-1}{2},\frac{N-1}{2}\right)
\hbox{ \ \ for }  t=1,\dots,q-1\\
&&z_{q}=i_{s_q-1}-\left(\frac{N-1}{2},\frac{N-1}{2}\right).
\end{eqnarray*}
the result follows as in the previous case.
\end{proof}

\subsection{Geometry of minimizers}
We conclude by stating the results concerning the geometry of the ground states, corresponding to Proposition \ref{path_bis} and Theorem \ref{th_rand} respectively.
The main result states that for
spin systems not in ${\mathcal B}^{1/2}_p$ the minimizers of $F$ on large sets
are characterized by a majority phase. Remark \ref{trans} then assures that
this is a generic situation for $N$ large.

\begin{theorem} Let $N\in{\mathbb N}$, and let $\{c_{ij}\}$ be a $N$-periodic
distribution of ferro/antiferro\-magnetic interactions
such that $\{c_{ij}\}\not\in\tilde{\mathcal B}^{1/2}_p$.
Let $G$ be a Lipschitz bounded open set and let
$u_\e$ be a minimizer for $F(\cdot,{1\over\e}G)$.
Then there exists a constant $C$ depending only on $G$
such that either $\overline{\{u_\e=1\}}$ or $\overline{\{u_\e=-1\}}$ is composed of
connected components $K_\e^i$ such that
the length of the boundary of each $K_\e^i$ is not greater than $CN$.
\end{theorem}
As for Theorem \ref{th_rand}, the proof relies on the estimate of the length of paths in the interface, which in this case reads as follows.
\begin{proposition} \label{path}
Let $N\in{\mathbb N}$, and let $\{c_{ij}\}$ be a $N$-periodic
distribution of ferro/anti\-ferro\-magnetic interactions
such that $\{c_{ij}\}\not\in\tilde{\mathcal B}^{1/2}_p$.
Let $D$ be an open bounded subset of $\mathbb R^2$ such that 
the distance between the connected components of $\partial q(D)$ is greater than $N/2$.
Let $u$ be a minimizer for $F(\cdot,D)$.
Then there is no path in the interface $S(u)$ of length greater than $N/2$.
\end{proposition}
The steps of the proofs are exactly the same as in the random case, by substituting the applications of Lemma \ref{l_proba1} with the corresponding applications of Proposition \ref{globale} (thus the logarithmic estimates become linear with $N$).

\goodbreak

\end{document}